\documentclass{article}
\usepackage[letterpaper,top=2cm,bottom=2cm,left=3cm,right=3cm,marginparwidth=1.75cm]{geometry}

\usepackage{enumitem}   
\usepackage{xr}

\RequirePackage[OT1]{fontenc}
\RequirePackage{amsthm,amsmath,amssymb}
\RequirePackage[numbers]{natbib}
\RequirePackage[colorlinks,citecolor=blue,urlcolor=blue]{hyperref}

\newcommand{\beq}{\begin{equation}}
\newcommand{\eeq}{\end{equation}}

\newcommand{\PP}{\mathbb{P}}
\newcommand{\RR}{\mathbb{R}}
\newcommand{\NN}{\mathbb{N}}
\newcommand{\EE}{\mathbb{E}}
\newcommand{\ZZ}{\mathbb{Z}}

\newcommand{\test}{\phi}
\newcommand{\step}{s}

\newcommand{\radius}{r}

\theoremstyle{plain}
\newtheorem{theorem}{Theorem}
\theoremstyle{remark}
\newtheorem{remark}[theorem]{Remark}
\theoremstyle{plain}
\newtheorem{assumption}[theorem]{Assumption}
\theoremstyle{plain}

\theoremstyle{plain}
\newtheorem{lemma}[theorem]{Lemma}
\theoremstyle{plain}
\newtheorem{proposition}[theorem]{Proposition}
\theoremstyle{definition}

\begin{document}

\author{Thomas Bonis\\
	LAMA, Universit\'e Gustave Eiffel,  France \\
	thomas.bonis@univ-eiffel.fr}

\date{}
	\title{Stein's method for steady-state diffusion approximation in Wasserstein distance}

	\maketitle
 


  
   



\begin{abstract}
We provide a general steady-state diffusion approximation result which bounds the Wasserstein distance between the reversible measure $\mu$ of a diffusion process and the invariant measure $\nu$ of a Markov chain. Our result is obtained thanks to a generalization of a new approach to Stein's method which may be of independent interest. As an application, we study the convergence of the invariant measure of the random walk on $k$-nearest neighbors graphs, providing a quantitative answer to a problem of interest to the machine learning community.
\end{abstract}



\section{Introduction}

Consider a diffusion process $(X_t)_{t \geq 0}$ with generator $\mathcal{L}_\mu$ and reversible measure $\mu$. Conditions for an approximating family of Markov chains $(Y^n_t)_{t \geq 0}$ with stationary measure $\nu^n$ to converge to $(X_t)_{t \geq 0}$ are well-known, see e.g. Corollary 4.2 \citep{ethierkurtz}. However, such a convergence, called diffusion approximation, does not guarantee the convergence of the measures $\nu^n$ to $\mu$. This problem of steady-state diffusion approximation is at the center of multiple recent works \cite{Braverman1, Braverman3, Braverman2, Gurvich} as a way to study queuing systems. One of the main tool used by these works is Stein's method, which is a standard technique used to bound distances between measures. In particular, these work rely on the traditional approach to Stein's method which aims at solving the following equation, known as Stein's equation, 
\[
u - \int_{\RR^d} u \, d\mu = \mathcal{L}_\mu f_u,
\]
for $u$ belonging to a specific set of functions $\mathcal{U}$. Then, taking the integral over $\nu$ yields 
\[
\int_{\RR^d} u \, d\nu - \int_{\RR^d} u \, d\mu = \int_{\RR^d} \mathcal{L}_\mu f_u \, d\nu.
\]
Thus, bounds on
\[
\sup_{u \in \mathcal{U}} \left|\int_{\RR^d} \mathcal{L}_\mu f_u  \, d\nu \right|,
\]
which are known as Stein Factor bounds, translate to bounds on
\[
\sup_{u \in \mathcal{U}} \left|\int_{\RR^d} u \, d\nu - \int_{\RR^d} u \, d\mu \right|. 
\]
By choosing an appropriate set $\mathcal{U}$, one is then able to bound  various distances between $\mu$ and $\nu$:
\begin{itemize}
\item if $\mathcal{U} = \{ u : \RR^d \rightarrow \RR \mid \|u\|_\infty \leq 1 \}$, then $\sup_{u \in \mathcal{U}} |\int_{\RR^d} u \, d\nu - \int_{\RR^d} u \, d\mu|$ is the total variation distance between $\mu$ and $\nu$;
\item in dimension $1$, if $\mathcal{U} = \{  u : \RR \rightarrow \RR \mid \exists t \in \RR, u(x) = 1_{x \leq t}  \}$, then $\sup_{u \in \mathcal{U}} |\int_{\RR}  u \, d\nu - \int_{\RR^d} u \, d\mu|$ is the Kolmogorov distance between $\mu$ and $\nu$;
\item if $\mathcal{U} = \{  u : \RR^d \rightarrow \RR \mid \forall x,y \in \RR^d,  |u(y) - u(x)| \leq \|y-x\|  \}$, where $\|.\|$ denotes the Euclidean norm, then $\sup_{u \in \mathcal{U}} |\int_{\RR^d} u \, d\nu - \int_{\RR^d} u \, d\mu|$ is the Wasserstein distance of order $1$, also called Kolmogorov-Rubinstein distance, between $\mu$ and $\nu$. 
\end{itemize}
However, deriving such Stein Factor bounds is often difficult in practice, especially in a multivariate setting. 

Recently, \cite{Stein} introduced another approach to Stein's method which does not rely on solving Stein's equation and can be used to tackle general target measures $\mu$ by working with the framework of Markov Triple introduced in \cite{Markov}. This setting is particularly well-suited for our problem as it corresponds to assuming that the target measure $\mu$ is the reversible measure of a generator $\mathcal{L}_\mu$ with semigroup $(P_t)_{t \geq 0}$. 
However, the results obtained in \cite{Stein} require the existence of a Stein kernel for the measure $\nu$ which usually requires $\nu$ to be a continuous measure in the first place. Moreover, this approach also requires strong regularization properties for $(P_t)_{t \geq 0}$, obtained at the cost of further assumptions on $\mu$ which are unlikely to be met even in simple cases, see Section~\ref{sec:regularization} for more details about this matter. 

Whenever the target measure is the Gaussian measure, \cite{Bonis} managed to use the previous approach without relying on a Stein kernel. In fact, the result obtained is quite similar to a steady-state diffusion approximation bound. In this work, we wish to extend this approach to more general target measures by working with the Markov Triple framework used in \cite{Stein}. The main ingredient to generalizing the approach of \cite{Bonis} in this setting are regularization properties for the diffusion semigroup $(P_t)_{t \geq 0}$, which we obtain under a simple curvature-dimension inequality in Section~\ref{sec:regularization}. This allow us to provide a general quantitative steady-state diffusion approximation result, expressed in terms of Wasserstein distance of order $2$, which is presented in Section~\ref{sec:main}. This result is quite natural as it only involves a quantified version of the generator convergence assumption at the center of traditional diffusion approximation results. 
In Section~\ref{sec:graph}, we apply this result to tackle a problem from the machine learning community introduced in \cite{Hashimoto} by studying the convergence of invariant measures on $k$-nearest neighbor graphs. 

\section{Notations}
\label{sec:notations}
Let $d$ be a positive integer. 
For any $k \in \NN$, let $(\RR^d)^{\otimes k}$ be the set of tensors of order $k$, that is elements of the form $(x_j)_{j \in \{1,\dots,d\}^k} \in \RR^{d^k}$.
For $x \in \RR^d$ and $k \in \NN$, we denote by $x^{\otimes k}$ the tensor such that 
\[
\forall j \in \{1,\dots,d\}^k, (x^{\otimes k})_{j} = \prod_{i=1}^k x_{j_i}.
\]
For any $x,y \in (\RR^d)^{\otimes k}$ and any symmetric positive-definite $d \times d$ matrix $A$, let 
\[
<x,y>_A = \sum_{l ,j \in \{1,\dots,d\}^k} \left( x_l y_j \prod_{i=1}^k A_{l_i,j_i} \right),
\]
and, by extension, 
\[
\|x\|^2_A = <x,x>_A.
\]
Finally, we denote by $<.,.>$ the traditional Hilbert-Schmidt scalar product, corresponding to $<.,.>_{I_d}$, and by $\|.\|$ its associated norm. 

For any two spaces $E,F \subset \RR^d$, we denote by $C^k(E,F)$ the set of functions from $E$ to $F$ with continuous partial derivatives of order $k \in \NN$ and we denote by 
$C^k_c(E,F)$ the set of such functions with compact support. For any $k \in \NN$, any function $\test\in C^k( \RR^d ,\RR)$ and any $x \in \RR^d$, we denote by $\nabla^k \test \in (\RR^{d})^{\otimes k}$ the $k$-th gradient of $\test$:
\[
\forall j \in \{1,\dots,d\}^k, (\nabla^k \test (x))_j = \frac{\partial^k \test}{\partial x_{j_1} \dots \partial x_{j_k} }(x).
\]

Consider a connected and open set $E \subset \RR^d$ and a matrix-valued function $a : E \rightarrow (\RR^d)^{\otimes 2}$ such that $a$ is positive-definite on all of $E$. For any $x\in E$, $a(x)$ admits an inverse matrix $a^{-1}(x)$.
We denote by $d_a$ the metric on $E$ induced by $a$ and defined by 
\[
\forall x,y \in E, d_a(x,y) = \inf_{\gamma} \int_0^1 \|\gamma'(t)\|_{a^{-1}(\gamma(t))} dt, 
\]
where the infimum is taken over all curves $\gamma \in C^1( [0,1], E)$ such that $\gamma(0) = x$ and $\gamma(1)= y$. 

Finally, given two probability measures $\mu$ and $\nu$ on $E$, we denote by $W_{2,a}$ the Wasserstein distance of order $2$ with respect to the metric $d_a$ and defined by 
\[
W_{2,a}^2 =  \inf_{\pi}  \int_{E \times E} d_a(x,y)^2 \pi(dx,dy),
\]  
where $\pi$ has marginals $\mu$ and $\nu$. Finally, we denote $W_{2,I_d}$ by $W_2$. 

\section{A general steady-state diffusion approximation result}
\label{sec:main}
\subsection{Statement}
Let $E$ be a domain of $\RR^d$ and consider two functions $a \in C^\infty(E, (\RR^d)^{\otimes 2})$ and $b \in C^\infty(E, \RR^d)$ such that 
$a(x)$ is symmetric and positive-definite for any $x \in E$. Let $\mu$ be the reversible probability measure of a Markov process $(P_t)_{t \geq 0}$ with infinitesimal generator $\mathcal{L}_\mu$ defined by
\[
\forall \test \in C^\infty(E, \RR), x \in E, \mathcal{L}_\mu f(x) = b(x). \nabla f(x) + <a(x), \nabla^2 f(x)>.
\]
We denote by $\Gamma_1$ the carr\'e du champ operator defined by
\[
\forall f,g \in C^\infty(E, \RR), \Gamma_1(f,g) = <\nabla f, \nabla g>_a
\]
and  by $\Gamma_2$ the operator defined by 
\beq
\label{eq:gamma2}
\forall f,g \in C^\infty(E,\RR), \Gamma_2(f,g) = 
\frac{1}{2} \left[ \mathcal{L}_\mu (\Gamma_1(f,g)) - \Gamma_1(\mathcal{L}_\mu f, g) - \Gamma_1( f, \mathcal{L}_\mu g) \right].
\eeq
The triple $(E, \mu, \Gamma_1)$ then forms a Markov Triple, a structure which is extensively studied in \cite{Markov}. Let us note that, while we restrict ourselves by considering only measures on $\RR^d$ for simplicity, the framework of Markov Triples is also suited to deal with measures supported on manifolds.
We assume that the Markov Triple $(E, \mu, \Gamma_1)$ verifies a curvature-dimension condition $CD(\rho,\infty)$. That is, we assume there exists $\rho \in \RR$ such that 
\beq
\label{eq:CD-cond}
\forall \test \in C^\infty(E, \RR), \Gamma_2(\test,\test) \geq \rho \Gamma_1(\test,\test).
\eeq
Under this assumption, the semigroup $(P_t)_{t \geq 0}$ enjoys many regularizing properties which will prove crucial in our approach, see Section~\ref{sec:regularization} for more details. 
We also assume that, for any measure $\eta$ such that $d \eta = h d \mu$, the measure $\eta_t$ with $d\eta_t = P_t h \, d\mu$ converges exponentially fast to $\mu$. More precisely, we assume there exists $c  \geq 1, \kappa > 0$ such that 
\[
\forall t > 0, W_{2,a}(\nu_t,\mu) \leq c e^{-\kappa t}  W_{2,a}(\nu,\mu). 
\]
While such an exponential convergence to $\mu$ is verified whenever $\rho > 0$ with $\kappa = \rho$ and $c = 1$ (see Theorem 9.7.2 \citep{Markov}), it can also be obtained under weaker assumptions. For example,  
if $a = I_d$ and $b = -\nabla V$, where $V \in C^\infty(\RR^d, \RR)$ is a potential, this property is satisfied whenever $V$ is strongly convex outside a bounded set $C$ 
and has bounded first and second order derivatives 
on $C$ \citep{GGB}. An extension of this result for more general functions $a$ is proposed in Theorem 2.1 \citep{Fwang}.

Let us summarize the assumptions made on $\mu$ so far. 
\begin{assumption}
\label{ass:main}
\begin{enumerate}[label=(\roman*)]
\item \label{ass:main-1} $\mu$ is the reversible probability measure of a diffusion operator $\mathcal{L}_\mu$ defined by
\[
\forall \test \in C^\infty(E, \RR^d), \mathcal{L}_\mu \test = b . \nabla \test + <a, \nabla^2 \test>,
\]
with $b \in C^\infty(E, \RR^d)$ and $a \in C^\infty(E, (\RR^d)^{\otimes 2})$ such that $a$ is symmetric positive definite on all of $E$ and $\|b\|_{a^{-1}} \in L_1(\mu)$. Moreover, $\mathcal{L}_\mu$ is the generator of a semigroup $(P_t)_{t \geq 0}$ which is symmetric with respect to the measure $\mu$.
\item \label{ass:main-2} There exists $\rho \in \RR$ such that 
\[
\forall \test \in C^\infty(E, \RR), \Gamma_2(\test,\test) \geq \rho \Gamma_1(\test,\test). 
\]
\item \label{ass:main-3} We have $d_a(.,0) \in L_2(\mu)$. Furthermore, there exists $c \geq 1,\kappa > 0$ such that, for any probability measure $\eta$ with $d_a(.,0) \in L_2(\eta)$ and $d\eta = h d\mu$, 
 \[
 W_{2,a}(\eta_t,\mu) \leq c e^{-\kappa t}  W_{2,a}(\eta,\mu),
 \]
 where $\eta_t$ has measure $P_t h$.
 \end{enumerate}
\end{assumption}
We now consider another measure $\nu$ on $E$ assumed to be the invariant measure of a Markov kernel $K$. In order to present the assumptions required on $\nu$, let us introduce the set of functions $(f_k)_{k \in \NN^\star}$ where, for any $k \in \NN^\star$,
\[
f_k(t) = \begin{cases} 
 e^{- \rho t \max(1,k/2)} \left( \frac{\rho d}{e^{2\rho t/(k-1)}-1} \right)^{(k-1)/2}  \textit{ if } \rho \neq 0 \\
\left(\frac{d (k-1)}{2t}\right)^{(k-1)/2} \textit{ if } \rho = 0 
\end{cases}.
\]
The assumptions we require on $K$ are as follows 
\begin{assumption}
\label{ass:mainnu}
\begin{enumerate}[label=(\Roman*)]
\item \label{ass:mainnu1} $\nu$ is an invariant probability measure for the Markov kernel $K$ and $d_a(.,0) \in L_2(\nu)$.
\item \label{ass:mainnu2} For any $x,y$ for which $K(x,dy) > 0$, $[x,y]$ belongs to $E$.
\item \label{ass:mainnu3} 
There exists $\tau, T > 0$ such that, 
\[
 \int_{E \times E} \sup_{t \in [\tau,T]} \left(\sum_{k=1}^\infty \frac{f_k(t)}{k!} \left\|y-x\right\|^{k}_{a^{-1}(x)} \right)^{2}  K(x,dy) \, \nu(dx) < \infty. 
\]
\end{enumerate}
\end{assumption}
Such assumptions are light in practice, the most restricting one being assumption \ref{ass:mainnu3}. In order to give some intuition about this assumption, let us note that, whenever $a = I_d$ and $\rho = 0$, we can use Cauchy-Schwarz inequality and Stirling's approximation to obtain
\begin{align*}
\int_{E} \sup_{t \in [\tau,T]} & \left(\sum_{k=1}^\infty \left(\frac{d (k-1)}{2t}\right)^{(k-1)/2} \frac{1}{k!} \left\|y-x\right\|^{k} \right)^{2}  K(x,dy)  \\
& \leq 2 \left( \sum_{k=1}^\infty \frac{1}{2^{k}} \right) \int_{E}  \left(\sum_{k=1}^\infty \left(\frac{d (k-1)}{\tau}\right)^{k-1} \frac{1}{k!^2}\left\|y-x\right\|^{2 k} \right)  K(x,dy) \\
& \leq \sqrt{\frac{2}{\pi}}\frac{\tau}{d}  \int_{E} \sum_{k=1}^\infty \frac{1}{k!} \left(\frac{de \left\|y-x\right\|^{2}}{\tau}\right)^k   K(x,dy) \\
& \leq \sqrt{\frac{2}{\pi}}\frac{\tau}{d} \int_{E} \left(e^{\frac{de \left\|y-x\right\|^{2}}{\tau}} - 1\right) K(x,dy).
\end{align*}
Assumption \ref{ass:mainnu3} is thus satisfied whenever $K$ has a Gaussian tail in this simple case. 

We are now ready to state the main result of this paper.
\begin{theorem} 
\label{thm:main}
Suppose Assumption~\ref{ass:main} and \ref{ass:mainnu} are verified. Then, for any $\step > 0$, there exists $C(T, \rho, d) > 0$ such that, for any $s > 0$,
\begin{align*}
\frac{W_{2,a}(\nu, \mu) }{C(T, \rho, d) }  & \leq  \tau \left(\int_E  \|b(x)\|^2_{a^{-1}(x)} d\nu(x)\right)^{1/2} + \sqrt{\tau} \\
&    + \left(\int_E \left\|\frac{1}{\step} \int_E (y-x) K(x,dy)  - b(x) \right\|^2_{a^{-1} (x)} d\nu(x) \right)^{1/2}  \\
 & +  \left(\int_E \left\| \frac{1}{2 \step} \int_{E} (y-x)^{\otimes2} K(x,dy)  - a(x) \right\|^2_{a^{-1}(x)} d\nu(x) \right)^{1/2}\\
 & + \frac{\log(\tau)}{\step} \left(\int_E \left\|  \int_{E} (y-x)^{\otimes 3} K(x,dy)\right\|^2_{a^{-1}(x)} d\nu(x) \right)^{1/2} \\
 & + \sum_{k=4}^\infty \frac{C(T, \rho, d) ^{k-1} }{\step \sqrt{k! \tau^{k-3}}} \left( \left\|  \int_{E} (y-x)^{\otimes k} K(x,dy)\right\|^2_{a^{-1}(x)} d\nu(x)\right)^{1/2}. 
\end{align*}
\end{theorem}
Let us remark that the quantities appearing in our bound are natural as they correspond to a quantification of the usual generator convergence condition appearing in standard diffusion approximation results, see e.g. Corollary 4.2 \citep{ethierkurtz}. 

\subsection{Proof}
Suppose Assumptions~\ref{ass:main} and \ref{ass:mainnu} are verified.
Furthermore, let us assume for now that the measure $\nu$ admits a density $h$ with respect to $\mu$ such that $h = \epsilon + f$ for some $\epsilon > 0$ and $f \in C^\infty_c(E,\RR^+)$ and that $K(x,.)$ has bounded support for all $x \in E$. These assumptions can be later lifted thanks to an approximation argument detailed in Section~\ref{sec:approx}. 

\subsubsection{Diffusion interpolation}
For any $t > 0$, let $\nu_t$ be the measure with density $P_t h$. Since  $\nu_t$ converges to $\mu$ as $t$ grows to infinity, we can bound the distance between $\nu$ and $\mu$ by controlling the infinitesimal displacements of $\nu_t$. This idea is formalized in the following result.

\begin{lemma}
\label{lem:interp}
Suppose Assumption~\ref{ass:main} is verified and that the measure $\nu$ admits a density $h$ with respect to $\mu$ such that $h = \epsilon + f$ for some $\epsilon > 0$ and $f \in C^\infty_c(E,\RR^+)$. Then, for any $T > 0$, 
\[
(1 - c e^{- \kappa T}) W_{2,a}(\nu,\mu) \leq \int_0^T \left( \int_E (-\mathcal{L}_\mu) P_t v_t \, d\nu(x) \right)^{1/2} dt,
\]
where $v_t = \log P_t h$.
\end{lemma}

\begin{proof}
Let $T > 0$. By \ref{ass:main-3}, we have 
\begin{align*}
W_{2,a}(\nu,\mu) & \leq W_{2,a}(\nu, \nu_T) + W_{2,a}(\nu_T, \mu) \\
& \leq  W_{2,a}(\nu, \nu_T) + c e^{-\kappa T} W_{2,a}(\nu, \mu),
\end{align*}
from which we obtain
\beq
\label{eq:expoconv}
(1 - c e^{- \kappa T}) W_{2,a}(\nu,\mu) \leq W_{2,a}(\nu, \nu_T).
\eeq
It it thus sufficient to bound $W_{2,a}(\nu, \nu_T)$ in order to bound $W_{2,a}(\nu,\mu)$.
Such a bound can be derived from the following estimate, obtained in Equation (3.4) \cite{WangOV}, 
\beq
\label{eq:OV}
\forall t > 0, \frac{d^+}{dt} W_{2,a}(\nu, \nu_t) \leq \left(\int_E \left< \nabla \log P_t h(x), \nabla P_t h(x) \right>_{a(x)} d \mu(x) \right)^{1/2} = I_\mu(\nu_t)^{1/2},
\eeq
where $I_\mu(\nu_t)$ is the Fisher information of the measure $\nu_t$ with respect to $\mu$. 
Let $t > 0$ and $v_t = \log P_t h$. We have
\begin{align*}
P_t h \mathcal{L}_\mu v_t & = P_t h( b.\nabla v_t + <a, \nabla^2 v_t>) \\
& = b . \nabla P_t h + \left<a, \nabla^2 P_t h - \frac{(\nabla P_t h)^{\otimes 2}}{P_t h} \right> \\
& = \mathcal{L}_\mu P_t h - \frac{\|\nabla P_t h\|_{a}^2}{P_t h} \\
& = \mathcal{L}_\mu P_t h - \left<\nabla v_t, \nabla P_t h\right>_a.
\end{align*}
Since $h = \epsilon + f$ with $f \in C^\infty_c$, $P_t h$ and $\| \nabla v_t \|_a$ are bounded. Thus, there exists $C > 0$ such that $|\mathcal{L}_\mu P_t h| \leq C (\|b\|_{a^{-1}} + 1)$. By  \ref{ass:main-1}, $\|b\|_{a^{-1}} \in L_1(\mu)$, thus
$\mathcal{L}_\mu P_t h \in L_1(\mu)$. Moreover,
$\mu$ is an invariant measure of the operator $\mathcal{L}_\mu$, which means that 
\[
\int_{E} \mathcal{L}_\mu P_t h (x) \, d\mu(x) = 0.
\]
Therefore,
\[
I_\mu(\nu_t) = -\int_{E} P_t h(x) \mathcal{L}_\mu v_t(x) \, d\mu(x).
\]
Then, by the symmetry of $(P_t)_{t \geq 0}$ with respect to the measure $\mu$, we have
\[
I_\mu(\nu_t)= - \int_{E} h(x) P_t \mathcal{L}_\mu v_t(x) \, d\mu(x) = - \int_{E} P_t \mathcal{L}_\mu v_t(x) \, d\nu(x).
\]
Finally, since $\mathcal{L}_\mu$ is the infinitesimal generator of the semigroup $(P_t)_{t \geq 0}$, we can permute $P_t$ and $\mathcal{L}_\mu$ to obtain 
\[
I_\mu(\nu_t) = -\int_{E} \mathcal{L}_\mu P_t v_t(x) \, d\nu(x),
\]
concluding the proof of Lemma~\ref{lem:interp}. 
\end{proof}

\subsubsection{Generator comparison}

Let $\step > 0$ be a rescaling factor and let $\mathcal{L}_\nu$ be the rescaled generator associated with the Markov kernel $K$, defined for any bounded function $\phi$ by
\[
\forall x \in E, \mathcal{L}_\nu \test (x) = \frac{1}{\step}  \int_E (\test(y) - \test(x)) K(x, dy)
\]
By \ref{ass:mainnu1}, we have 
\[
\int_E \mathcal{L}_\nu \test (x) \, d\nu(x) = 0.  
\]
Hence, by Lemma~\ref{lem:interp} we have
\[
(1 - c e^{- \kappa T}) W_{2,a}(\nu,\mu) \leq \int_0^T I_\mu(\nu_t)^{1/2} \, dt,
\]
with
\[
I_\mu(\nu_t) = \int_E (-\mathcal{L}_\mu) P_t v_t(x) \, d\nu(x) = \int_E (\mathcal{L}_\nu-\mathcal{L}_\mu) P_t v_t(x) \, d\nu(x).
\]
Let $x \in E$. By Lemma~\ref{lem:analytic}, we have that $P_t v_t$ is real analytic on $E$. Thus, by \ref{ass:mainnu2} and since $K(x,.)$ has bounded support, 
\[
\mathcal{L}_\nu P_t v_t (x) = \sum_{k=1}^\infty \frac{1}{\step k!} \left<\int_E (y - x)^{\otimes k} K(x, dy), \nabla^k P_t v_t(x)\right>.
\]
Therefore,
\begin{align*}
(\mathcal{L}_\nu-\mathcal{L}_\mu) P_t v_t(x) = &  \left<\frac{1}{\step} \int_E (y - x) K(x, dy) - b(x), \nabla P_t v_t(x)\right>  \\
& + \left<\frac{1}{2 \step} \int_E (y - x)^{\otimes 2} K(x, dy) - a(x), \nabla^2 P_t v_t(x)\right> \\
& + \sum_{k=3}^\infty \frac{1}{\step k!} \left<\int_E (y - x)^{\otimes k} K(x, dy), \nabla^k P_t v_t(x)\right>
\end{align*}
and applying Cauchy-Schwarz inquality yields 
\begin{align*}
(\mathcal{L}_\nu-\mathcal{L}_\mu) P_t v_t(x) \leq &  \left\|\frac{1}{\step} \int_E (y - x) K(x, dy) - b(x)\right\|_{a^{-1}(x)} \|\nabla P_t v_t(x)\|_{a(x)} \\
& + \left\|\frac{1}{2\step} \int_E (y - x)^{\otimes 2} K(x, dy) - a(x)\right\|_{a^{-1}(x)}  \left\|\nabla^2 P_t v_t(x)\right\|_{a(x)} \\
& + \sum_{k=3}^\infty \frac{1}{\step k!}\left\|\int (y - x)^{\otimes k} K(x, dy)\right\|_{a^{-1}(x)} \left\| \nabla^k P_t v_t(x)\right\|_{a(x)} .
\end{align*}
From here, using Proposition~\ref{pro:curvdim} gives
\[
(\mathcal{L}_\nu-\mathcal{L}_\mu) P_t v_t(x)  \leq S(x,t) \left(P_t \|\nabla v_t(x)\|^2_{a(x)}\right)^{1/2},
\]
where
\begin{align}
\label{eq:Sdef}
\begin{split}
S(x,t) = & f_1(t) \left\|\frac{1}{\step}\int_E (y - x) K(x, dy) - b(x)\right\|_{a^{-1}(x)}  \\
& + f_2(t) \left\|\frac{1}{2\step} \int_E (y - x)^{\otimes 2} K(x, dy) - a(x)\right\|_{a^{-1}(x)}  \\
& + \sum_{k=3}^\infty \frac{f_k(t)}{\step k!} \left\|\int (y - x)^{\otimes k} K(x,dy)\right\|_{a^{-1}(x)}
\end{split}
\end{align}
and 
\[
f_k(t) = \begin{cases} 
 e^{- \rho t \max(1,k/2)} \left( \frac{\rho d}{e^{2\rho t/(k-1)}-1} \right)^{(k-1)/2}  \textit{ if } \rho \neq 0 \\
\left(\frac{d (k-1)}{2t}\right)^{(k-1)/2} \textit{ if } \rho = 0 
\end{cases}.
\]
Then, using Cauchy-Schwarz inequality, we obtain
\[
I_\mu(\nu_t) = \int_E (\mathcal{L}_\nu-\mathcal{L}_\mu) P_t v_t(x) \, d\nu(x) \leq \left(\int_E S(x,t)^2 \, d\nu(x) \right)^{1/2} \left(\int_E P_t \| \nabla  v_t(x)\|^2_{a(x)} \, d\nu(x) \right)^{1/2}
\]
and, since 
\[
\int_E P_t \| \nabla  v_t(x)\|^2_{a(x)} \, d\nu(x) = I_\mu(\nu_t),
\]
we have 
\[
I_\mu(\nu_t) \leq \left(\int_E S(x,t)^2 \, d\nu(x) \right)^{1/2} I_\mu(\nu_t)^{1/2}.
\]
Therefore,
\[
I_\mu(\nu_t)^{1/2} \leq \left(\int_E S(x,t)^2 \, d\nu(x)\right)^{1/2}.
\]
Now, using crude bounds on the functions $f_k$ derived in (\ref{eq:fkbound}), we can integrate the right-hand term of the previous equation between $\tau$ and $T$ to obtain 
\begin{align*}
\frac{\int_\tau^T  I_\mu(\nu_t)^{1/2} \, dt}{C(T,\rho,d)}  \leq
&    \left(\int_E \left\|\frac{1}{\step} \int_E (y-x) K(x,dy)  - b(x) \right\|^2_{a^{-1} (x)} d\nu(x) \right)^{1/2}  \\
 & +  \left(\int_E \left\| \frac{1}{2 \step} \int_{E} (y-x)^{\otimes2} K(x,dy)   - a(x) \right\|^2_{a^{-1}(x)} d\nu(x) \right)^{1/2}\\
 & + \frac{\log(\tau)}{\step} \left(\int_E \left\|  \int_{E} (y-x)^{\otimes 3} K(x,dy)\right\|^2_{a^{-1}(x)} d\nu(x) \right)^{1/2} \\
 & + \sum_{k=4}^\infty  \frac{C(T,\rho,d)^{k-1}}{\step \sqrt{\tau^{k-3}  k!}} \left( \left\|  \int_{E} (y-x)^{\otimes k} K(x,dy)\right\|^2_{a^{-1}(x)} d\nu(x)\right)^{1/2},
\end{align*}
with $C(T,\rho,d) > 0$.
Then, performing the same computations with $\mathcal{L}_\nu = 0$, we obtain 
\[
I_\mu(\nu_t)^{1/2} \leq  f_1(t) \left(\int_E \left\| b(x)\right\|^2_{a^{-1}(x)}d\nu(x)\right)^{1/2} + f_2(t) \sqrt{d}
\]
and integrating this bound for $t \in [0,\tau]$ gives 
\[
\int_0^\tau  I_\mu(\nu_t)^{1/2} dt \leq  C(\rho, d) \left( \tau \left(\int_E  \|b(x)\|^2_{a^{-1}(x)} d\nu(x)\right)^{1/2} +  \sqrt{\tau}\right),
\]
where $C(\rho, d)$ is a strictly positive constant. Combining both bounds to integrate $I_\mu(\nu_t)^{1/2}$ between $0$ and $T$ concludes the proof of Theorem~\ref{thm:main}.

\section{Regularization properties of Markov semigroups under a curvature-dimension inequality}
\label{sec:regularization}

Our objective in this Section is to provide bounds on $\|\nabla^k P_t \test\|_{a}$ for $\test \in C^\infty(E, \RR)$ and $k \in \NN^\star$. First, since we assume there exists $\rho \in \RR$ such that 
\beq
\label{eq:CDcondition}
\forall \test \in C^\infty_c(E, \RR), \Gamma_2(\test,\test) \geq \rho \Gamma_1(\test,\test),
\eeq
we know, by Theorem 3.2.3 \citep{Markov}, that 
$(P_t)_{t \geq 0}$ verifies the following gradient bound 
\[
\forall \test \in C^\infty_c(E,\RR), \|\nabla P_t \test\|_a \leq e^{- \rho t} P_t  \|\nabla \test\|_a.
\]
A bound on $\|\nabla^2 P_t v_t\|_a$ is obtained in a similar way in the course of the proof of Theorem 4.1 in \citep{Stein}. More precisely, by making use of the Gamma operators $(\Gamma_k)_{k \geq 1}$ which are defined recursively from the $\Gamma_1$ operator by 
\[
\forall k > 1, f,g \in C^\infty(E, \RR), \Gamma_{k}(f,g) = 
\frac{1}{2} \left[ \mathcal{L}_\mu (\Gamma_{k-1}(f,g)) - \Gamma_{k-1}(\mathcal{L}_\mu f, g) - \Gamma_{k-1}( f, \mathcal{L}_\mu g) \right],
\]
one can show that, if there exists $\kappa, \sigma > 0$ such that, for any $\test \in C^\infty_c(E, \RR)$, $\Gamma_3 (\test, \test)\geq \kappa \Gamma_2(\test, \test)$ and $\Gamma_2 (\test, \test) \geq \sigma \Gamma_1(\phi, \phi)$, then 
\[
\forall \test \in C^\infty_c(E, \RR), \|\nabla^2 P_t \test\|_a^2 \leq \frac{\kappa}{\sigma(e^{\kappa t}-1)}  P_t \|\nabla \test\|^2_a.
\]
However, such assumptions are usually hard to check in practice. For instance, let us consider a simple one-dimensional example for which $\mathcal{L}_\mu \test = -u' \test' + \test''$. 
In this case, (\ref{eq:CDcondition}) is verified as long as $u'' \geq \rho$. On the other hand, following the computations of Section 4.4 \citep{Stein}, in order to have $\Gamma_3(\test,\test) \geq 3 c \Gamma_2(\test,\test)$ and $\Gamma_2(\test,\test) \geq c \|\test''\|_a$ for some $c > 0$, one requires
\[
 u^{(4)} - u' u^{(3)} + 2 (u'')^2 - 6 c u'' \geq 0
 \]
 and 
 \[
 3 (u^{(3)})^2 \leq 2(u'' -c ) (u^{(4)} - u' u^{(3)} + 2 (u'')^2 - 6 c u'' ).
 \]
Hence, even in this rather simple case, one requires strong assumptions in order to bound $\|\nabla^2 P_t \test\|_a$ and bounding $\|\nabla^k P_t \test\|_a$ for $k > 2$ in a similar manner would require even stronger assumptions. 
Instead, we rely on the following result which provides bounds on $\|\nabla^k P_t \test\|_a$ under a simple curvature-dimension condition. 
\begin{proposition}
\label{pro:curvdim} 
Let $t > 0$ and suppose \ref{ass:main-1} and \ref{ass:main-2} are verified. Then, for any bounded function $\test \in C^\infty(E, \RR)$ such that $\|\nabla \test\|_a$ is bounded, we have 
\[
\forall k > 0, t > 0, \|\nabla^k P_t \test\|_a \leq f_k(t) \sqrt{P_t \|\nabla \test\|^2_a}, 
\]
where 
\[
f_k(t) = \begin{cases} 
 e^{- \rho t \max(1,k/2)} \left( \frac{\rho d}{e^{2\rho t/(k-1)}-1} \right)^{(k-1)/2}  \textit{ if } \rho \neq 0 \\
\left(\frac{d (k-1)}{2 t}\right)^{(k-1)/2} \textit{ if } \rho = 0 
\end{cases}.
\]
\end{proposition}

\begin{remark}
\label{rem:suboptimal}
The bounds we obtain are not dimension-independent as one could expect from the equivalent result for the Ornstein-Uhlenbeck semigroup obtained in (17) \citep{Bonis}. We believe this dependency to be an artifact of the proof.
\end{remark}

Such a result implies strong regularity for $P_t \test$. In particular, it must be real analytic on $E$. 
\begin{lemma}
\label{lem:analytic}
Let $\test \in C^\infty(E, \RR)$ be a bounded function and such that $\|\nabla \test\|_a$ is bounded. Then, under \ref{ass:main-1} and \ref{ass:main-2}, $P_t \test$ is real analytic on $E$ for any $t > 0$. Thus, for any $t > 0$ and any $x,y \in E$ such that $[x,y] \in E$, we have 
\[
P_t \test(y) - P_t \test(x) = \sum_{k=1}^\infty \frac{1}{k!} \left<(y-x)^{\otimes k} , \nabla^k P_t \test(x)\right>. 
\]
\end{lemma}

\subsection{Proof of Proposition~\ref{pro:curvdim}}
\label{sec:curvdim}

We are going to prove Proposition~\ref{pro:curvdim} by induction for the case $\rho \neq 0$, the case $\rho = 0$ can be obtained in a similar manner. First, by Theorem 3.2.3 \citep{Markov}, the result is true for $k = 1$. 
Now, let $k \geq 1$ and suppose that
\[
\forall t > 0,  \|\nabla^k P_t \test\|_a \leq e^{-\rho t \max(1, k/2)} \left( \frac{\rho d}{e^{(2 \rho t)/(k-1)} - 1}\right)^{(k-1)/2} (P_t \|\nabla \test\|_a^2)^{1/2}.
\]
Let $x \in \RR^d$ and 
let $(e_1,\dots,e_d)$ be an orthonormal basis of $\RR^d$ with respect to the scalar product $<.,.>_{a(x)}$. 
We have 
\begin{align*}
\|\nabla^{k+1} P_t \test(x)\|_{a(x)}^2 & =  \sum_{i=1}^d \sum_{j \in \{1,\dots,d\}^{k}}  \left<\nabla^{k+1} P_t \test(x), e_i \otimes e_{j_1} \otimes \dots \otimes e_{j_k} \right>_{a(x)}^2 \\
& = \sum_{i=1}^d \sum_{j \in \{1,\dots,d\}^{k}}   \left<\nabla^{k} < \nabla P_t \test(x),  e_i>_{a(x)}, e_{j_1} \otimes \dots \otimes e_{j_k} \right>_{a(x)}^2 \\
& = \sum_{i=1}^d \|\nabla^{k} <\nabla P_t \test(x), e_i>_{a(x)}\|_{a(x)}^2 \\
& = \sum_{i=1}^d   \left\|\nabla^k \lim_{\epsilon \rightarrow 0} \frac{P_t \test (x + \epsilon a(x) e_i) - P_t\test (x)}{\epsilon} \right\|_{a(x)}^2,
\end{align*}
leading to
\beq
\label{eq:curvdimaux}
\|\nabla^{k+1} P_t \test(x)\|_{a(x)}^2 
= \sum_{i=1}^d  \lim_{\epsilon \rightarrow 0} \left\|\nabla^k \frac{P_t \test (x + \epsilon a(x) e_i) - P_t\test (x)}{ \epsilon}\right\|_{a(x)}^2.
\eeq
Let $\epsilon > 0$ and let $(X_t)_{t \geq 0}$ and $(\tilde{X}^\epsilon_t)_{t \geq 0}$ be two diffusion processes with infinitesimal generator $\mathcal{L}_\mu$ and started respectively at $x$ and $x+\epsilon a(x) e_1$. Letting  $\psi_\epsilon : y \rightarrow \EE[\test(\tilde{X}^\epsilon_t) \mid X_t =y]$, we have
\begin{align*}
P_t \test (x + \epsilon a(x) e_1) - P_t\test (x) & = \EE[\test(\tilde{X}^\epsilon_t) - \test(X_t)] \\
& = \EE[ \EE[\test(\tilde{X}^\epsilon_t) \mid X_t] - \test(X_t)] \\
& = P_t(\psi_\epsilon - \test)(x).
\end{align*}
Using the Markov property of the semigroup $(P_t)_{t \geq 0}$ along with our induction hypothesis yields 
\begin{multline*}
\left\|\nabla^k P_t(\psi_\epsilon - \test)(x) \right\|^2_{a(x)} = \left\|\nabla^k P_{t (k-1)/k} P_{t/k} (\psi_\epsilon - \test)(x) \right\|^2_{a(x)} \\
\leq  e^{- \rho t \max(2,k) \frac{k-1}{k}} \left( \frac{\rho d}{e^{2 \rho  t/{k}} - 1} \right)^{k-1} P_{t(k-1)/k} \left\|\nabla P_{t/k}(\psi_\epsilon - \test)(x) \right\|^2_{a(x)}.
\end{multline*}
Then, using Theorem 4.7.2~\cite{Markov} and Jensen's inequality, we obtain
\begin{align*}
\left\|\nabla^k P_t (\psi_\epsilon - \test)(x)\right\|^2_{a(x)} & \leq 
 e^{-\rho t (k-1)} d^{k-1} \left( \frac{\rho }{e^{2 \rho t/k} - 1} \right)^{k} P_t \left|\psi_\epsilon - \test \right|^2(x) \\
 & \leq e^{-\rho t (k-1)} d^{k-1} \left( \frac{\rho }{e^{2 \rho t/k} - 1} \right)^{k} \EE\left[|\test(\tilde{X}^\epsilon_t) - \test(X_t)|^2\right].
\end{align*}
By Theorem 3.2.4 \citep{Markov} and Theorem 2.2 \citep{duality}, we can take $\tilde{X}^\epsilon_t$ such that, $d_a(\tilde{X}^\epsilon_t,X_t) \leq e^{-\rho t} d_a(\tilde{X}^\epsilon_0, X_0)$ almost surely. 
Using such $\tilde{X}^\epsilon_t$ and since $\|\nabla \test\|_a$ is bounded, we have 
\[
\left|\frac{\test(\tilde{X}^\epsilon_t) -  \test(X_t) }{\epsilon}\right| \leq C \frac{d_a(\tilde{X}^\epsilon_t, X_t) }{\epsilon} \leq C e^{-\rho t}\frac{d_a(\tilde{X}^\epsilon_0, X_0) }{\epsilon},
\]
where $ C = \sup \|\nabla \phi\|_a$. Then, since $a$ is continuous, we have that there exists $\epsilon_0$ such that, if $\epsilon < \epsilon_0$, then 
\[
d_a(\tilde{X}^\epsilon_0, X_0) \leq 2 \epsilon. 
\]
Therefore, we can apply the dominated convergence theorem to obtain
\begin{align*}
\lim_{\epsilon \rightarrow 0} \EE\left[ \left|\frac{ \test(\tilde{X}^\epsilon_t) -  \test(X_t) }{\epsilon} \right|^2 \right]
& =  \EE\left[ \lim_{\epsilon \rightarrow 0} \left|\frac{ \test(\tilde{X}^\epsilon_t) -  \test(X_t) }{\epsilon} \right|^2 \right] \\
& = \EE\left[\lim_{\epsilon \rightarrow 0} \left|\frac{<\tilde{X}^\epsilon_t - X_t,\nabla \test(X_t) >}{\epsilon}  \right|^2 \right] \\
& \leq \EE\left[ \|\nabla \test(X_t)\|_{a(X_t)}^2 \lim_{\epsilon \rightarrow 0} \frac{\|\tilde{X}^\epsilon_t - X_t\|_{a^{-1}(X_t)}^2}{\epsilon^2} \right] \\
& \leq \EE\left[\|\nabla \test(X_t)\|_{a(X_t)}^2 \lim_{\epsilon \rightarrow 0} \frac{d_a(\tilde{X}^\epsilon_t,X_t)^2}{\epsilon^2}   \right] \\
& \leq e^{-2 \rho t} \EE\left[\|\nabla \test(X_t)\|_{a(X_t)}^2 \lim_{\epsilon \rightarrow 0} \frac{d_a(\tilde{X}^\epsilon_0, X_0)^2}{\epsilon^2}  \right] \\
& \leq e^{-2 \rho t} \EE\left[\|\nabla \test(X_t)\|_{a(X_t)}^2 \right]\\
& \leq e^{-2 \rho t} P_t \|\nabla \phi\|_a^2.
\end{align*}
Since a similar result holds for all $(e_i)_{i \in \{1,\dots,d\}}$, combining this bound with (\ref{eq:curvdimaux}) yields 
\[
\|\nabla^{k+1} P_t \test\|^2_a \leq
e^{-\rho t (k+1)} \left( \frac{\rho d }{e^{2 \rho t/k} - 1} \right)^{k} \left( \sqrt{P_t \|\nabla \test\|_a^2} \right)^2,
\]
concluding the proof. 

\subsection{Proof of Lemma~\ref{lem:analytic}}
\label{sec:analproof}
Let us start by providing a crude bound on the functions $(f_k)_{k \geq 1}$. 
For $t> 0, k\in \NN$, let us recall that
\[
f_k(t) = \begin{cases} 
 e^{- \rho t \max(1,k/2)} \left( \frac{\rho d}{e^{2\rho t/(k-1)}-1} \right)^{(k-1)/2}  \textit{ if } \rho \neq 0 \\
\left(\frac{d (k-1)}{2 t}\right)^{(k-1)/2} \textit{ if } \rho = 0 
\end{cases}.
\]
First, if $\rho > 0$, we have 
\[
\frac{\rho}{e^{2 \rho t / (k-1)} - 1} \leq \frac{k-1}{2 t}.
\]
On the other hand, if $\rho < 0$, 
\[
\frac{\rho}{e^{2 \rho t / (k-1)} - 1} = e^{2 |\rho| t / (k-1)} \frac{|\rho|}{e^{2 |\rho| t / (k-1)} - 1} \leq e^{2 |\rho| t / (k-1)}\frac{k-1}{2t}.
\]
Thus, taking $D = \max(1, e^{- \rho t})$, 
\[
f_k(t) \leq D^{\max(1, k/2) + 1} \left(\frac{d (k-1)}{2t}\right)^{(k-1)/2} \leq  D^{2k} \left(\frac{d (k-1)}{2t}\right)^{(k-1)/2} .
\]
Then, by Stirling's approximation, we have that there exists $C > 0$ such that 
 \beq
 \label{eq:fkbound}
\forall k \geq 1, f_k(t) \leq \frac{C^k \sqrt{k!}}{t^{(k-1)/2}}.
\eeq
Now, let us consider a bounded function $\test$ such that $\|\nabla \test\|_a$ is bounded and $t > 0$. By Proposition~\ref{pro:curvdim}, we have, for any $x \in E$ and any $k \in \NN^\star$,
\[
\|\nabla^k P_t \test (x)\|_{a(x)} \leq f_k(t) \sqrt{P_t \|\nabla \phi\|_a^2}.
\]
Combining this with (\ref{eq:fkbound}) and since $\|\nabla \phi\|_a$ is bounded, we finally obtain that there exists $C > 0$ such that
\[
\|\nabla^k P_t \test (x)\|_{a(x)} \leq C'^k \sqrt{k!},
\]
for some $C' > 0$.
Since $a$ is symmetric positive definite on all of $E$, we can use Proposition 2.2.10 \cite{krantz2002primer}, to obtain that $P_t \test$ is real analytic on $E$ and conclude the proof.

\section{Invariant measure of random walks on nearest neighbours graphs}
\label{sec:graph}
\subsection{Context and statement}
Let $X_1, \dots, X_n$ be independent and identically distributed  random variables drawn from a measure $\mu$ with density $f$ on a manifold $\mathcal{T}$ of dimension $m$ embedded in $\RR^d$. 
We call random neighbourhood graph a graph $G_n$ with vertices $\mathcal{X}_n = (X_1,\dots,X_n)$ and edges $\{(x,y) \in \mathcal{X}^2 \mid \|x-y\|^2 \leq \radius_{\mathcal{X}_n}(x)\}$, where $\radius_{\mathcal{X}_n}: \RR^d \rightarrow \RR^+$ is a radius function.
In data analysis, it is common to analyze data by computing such neighbourhood graphs from the data and analyze this graph with graph analysis algorithms to perform various operations such as clustering, dimensional reduction and other data processing \citep{eigenmap,labelprop,spectral}.
With such an approach, one only relies on the properties of the random neighbourhood graph used while discarding all other information regarding the data. One may thus wonder whether all the relevant statistical information is contained in the graph or if some critical information is lost in the process. To answer this question, \citep{roadmap} proposed to check whether it is possible to estimate the density $f$ from which the data is drawn using only the structure of a random neighbourhood graph. Indeed, if one can compute a good estimator of $f$ from such a graph, then one can expect this graph to contain most of the relevant information regarding the original data. 

When considering undirected random neighborhood graph, such as geometric random graphs for which $\radius_{\mathcal{X}_n}$ is constant, the degree of the graph provides a simple and efficient estimator of $f$ which converges as $n \rightarrow \infty$ and $\radius_{\mathcal{X}_n} \rightarrow 0$ sufficiently slowly. However such an approach cannot be used for directed graphs which can have constant degree as we will see later on. Luckily, it is still possible to estimate $f$ using the invariant measure of a random walk on the graph, which corresponds to the degree function when the graph is undirected. As the number of points $n$ grows to infinity, it has been shown by \citep{Hashimoto}  that, if the radius function $\radius_{\mathcal{X}_n}$ converges, after a proper rescaling, to a deterministic function $\tilde{\radius} : \RR^d \rightarrow \RR^+$, 
then the invariant measures of random walks on the random neighbourhood graphs $G_n$ with radius functions $\radius_{\mathcal{X}_n}$ converge weakly to the invariant measure of a 
diffusion process with infinitesimal generator defined by
\[
\forall \test \in C^\infty(\RR^d, \RR), \mathcal{L}_{\tilde{\mu}} \test = \tilde{\radius}^2 \left( \nabla \log f . \nabla \test + \frac{1}{2} \Delta \test \right),
\]
where $\Delta$ is the Laplace operator. 
As the invariant measure $\tilde{\mu}$ of the limiting diffusion process has a density proportional to $\frac{f^2}{\tilde{\radius}^2}$, it is possible to derive an estimator of $f$ from the invariant measure of a random walk on the random neighborhood graph. 
Let us show how our results can be used to quantify this convergence by tackling the specific case of nearest neighbours graphs, which are quite popular in data analysis thanks to their sparsity. 

Nearest neighbours graphs are obtained by picking an integer $k > 0$ and putting an edge between two points $X_i$ and $X_j$ if and only if $X_j$ is one of the $k$-nearest neighbours of $X_i$. Equivalently, a $k$-nearest neighbours graph corresponds to a random neighbourhood graph with radius function 
\[
\radius_{\mathcal{X}_n}(x) = \inf \left\{r \in \RR^+ | \sum_{i =1}^n 1_{\|X_i - x\| \leq r} \geq k \right\}.
\]
Such graphs have constant degree equal to $k$, thus we cannot expect the degree function to be of any use to estimate $f$. However, if $k$ is properly chosen, random walks on such graphs are approximation of a diffusion process with infinitesimal generator 
\[
\forall \test \in C^\infty(\RR^d, \RR), \mathcal{L}_{\tilde{\mu}} \test = f^{-2/d} (\nabla \log f . \nabla \test +  \frac{1}{2}\Delta \test)
\]
and invariant measure $\tilde{\mu}$ with a density proportional to $f^{2+2/d}$. Our objective in this Section is to quantify the convergence of invariant measures of random walks on such graphs to $\tilde{\mu}$. To avoid boundary issues, let us assume that $\mu$ is supported on the flat torus $\mathcal{T} = (\RR / \ZZ)^d$ with strictly positive density $f \in C^\infty(\mathcal{T}, \RR^+)$. 
For any integer $k \leq n$, we denote by $\pi_{k,n}$ an invariant measure of a random walk 
on the $k$-nearest neighbour graphs with vertices $\mathcal{X}_n$. We obtain the following convergence result.

\begin{proposition}
\label{pro:randomgraph}
There exists $C > 0$ such that, for any positive integers $k,n$ such that $C \log(n) < k < n$, 
\[
\PP\left(W_2(\pi_{k,n}, \tilde{\mu}) \leq C\left(\sqrt{\frac{\log n}{k}} \left(\frac{n}{k}\right)^{1/d} + \left(\frac{k}{n}\right)^{1/d} \right) \right) \geq 1 - \frac{C}{n}.
\] 
\end{proposition} 
In particular, if $n >> k >> (\log n)^{d/(2+d)}  n^{2/(2+d)}$ then $W_2(\pi_{k,n}, \tilde{\mu})$ converges, in terms of Wasserstein distance of order $2$, to $\tilde{\mu}$. 
Let us note that this bound is likely to be suboptimal. Indeed, such a result is counterintuitive as the requirements on $k$ for the convergence to hold get weaker as the dimension of the data increases while we would expect the task of estimating $\tilde{\mu}$ to be more complex in higher dimensions. It is conjectured in \citep{Hashimoto} that it is sufficient for $n >> k >> \log(n)$ for $\pi_{k,n}$ to converge to $\tilde{\mu}$.

\subsection{Proof}

Let $\mathcal{T} = (\RR/\ZZ)^d$  be the $d$-dimensional flat torus and let $\mu$ be a measure supported on $\mathcal{T}$ with strictly positive density $f \in C^\infty(\mathcal{T},\RR^+)$. 
While $\mathcal{T}$ is not an open set of $\RR^d$, it is a flat and compact manifold. Thus, the arguments used in the proof of Theorem~\ref{thm:main} along with its conclusions still hold. 
Let $\tilde{\mu}$ be the measure with density $\tilde{f} = Z f^{2+2/d}$, where $Z > 0$ is a renormalization factor. Let us denote the Lebesgues measure on $\mathcal{T}$ by $\lambda$ and let $\nabla .$ be the divergence operator. For any two functions $\phi, \psi \in C^\infty(\mathcal{T}, \RR)$, we have, using an integration by parts with respect to the Lebesgues measure, 
\begin{align*}
\int_\mathcal{T} \phi f^{-2/d} (\nabla \log f . \nabla \psi & + \frac{1}{2} \Delta \psi ) d\tilde{\mu} \\
& = Z \int_\mathcal{T} \phi  (\nabla \log f . \nabla \psi  + \frac{1}{2} \Delta \psi ) f^2 d \lambda \\
& = \frac{Z}{2} \int_\mathcal{T} \phi  (\nabla \log f^2 . \nabla \psi  + \Delta \psi ) f^2 d \lambda \\
& = \frac{Z}{2} \int_\mathcal{T} \phi  (\nabla f^2 . \nabla \psi  + f^2 \Delta \psi )  d \lambda \\
& = \frac{Z}{2} \int_\mathcal{T} \phi  \nabla .(f^2 \nabla \psi) d \lambda \\
& = - \frac{Z}{2} \int_\mathcal{T} f^2 \nabla \phi  . \nabla \psi d \lambda \\
& = -\int_\mathcal{T} \frac{f^{-2/d}}{2} \nabla \phi  . \nabla \psi d \tilde{\mu}.
\end{align*}
The measure $\tilde{\mu}$ is thus the reversible probability measure of the operator $\mathcal{L}_{\tilde{\mu}}$ defined by
\[
\forall \test \in C^\infty(\mathcal{T}, \RR), \mathcal{L}_{\tilde{\mu}} \test = f^{-2/d}  (\nabla \log f . \nabla \test  + \frac{\Delta \test}{2}).
\] 
As $\mathcal{T}$ is compact, $f$ has bounded derivatives of all orders. Since $f$ is also strictly positive, $f^{-2/d} \nabla \log f$ and $f^{-2/d} I_d$ admit bounded derivatives of all orders and $f^{-2/d} I_d$ is strictly positive-definite on all of $\mathcal{T}$ and thus items \ref{ass:main-1} and \ref{ass:main-2} of Assumption~\ref{ass:main} are verified. 
Finally, thanks to Corollary 2.2 \cite{Fwang}, the Markov semigroup associated to $\mathcal{L}_{\tilde{\mu}}$ verifies property \ref{ass:main-3} of Assumption~\ref{ass:main}. 

Let $k,n \in \NN$ such that $k < n$, let $(X_1, \dots, X_n)$ be independent random variables drawn from the measure $\mu$ and let $\pi_{k,n}$ be the invariant measure of the random walk on the $k$-nearest neighbor graphs built on the point cloud $(X_1,\dots,X_n)$. For $x \in \mathcal{T}$ and $r > 0$, we denote by $B(x,r)$ the ball of radius $r$ centred in $x$. Finally, let
\[
\forall x \in \mathcal{T}, \radius_{\mathcal{X}_n}(x) = \inf \left\{s \in \RR^+ | \sum_{i =1}^n 1_{\|X_i - x\| \leq s} \geq k \right\}
\]
be the radius function corresponding to a $k$-nearest neighbor graph.
In the remainder of this Section, we denote by $C$ a generic constant depending only on $f$ and $d$. Furthermore, for any $m \in \NN$, let 
\[
V_m = \int_{\mathcal{B}(0,1)} x_1^m dx.
\]
By definition, $\pi_{k,n}$ is the invariant measure of the Markov chain with state space 
$\mathcal{X}_n$ and transition kernel $K$ defined on any $X_i, X_j \in \mathcal{X}_n$ by 
\[
K(X_i, X_j) = \frac{1}{k} 1_{\|X_j - X_i\| \leq  \radius_{\mathcal{X}_n}(x)}.
\]
As Assumption~\ref{ass:mainnu} is verified for the measure $\pi_{k,n}$ and the Markov kernel $K$ due to the compactness of $\mathcal{T}$, we can apply Theorem~\ref{thm:main} with $T = 1$ and 
\[
\tau = \step = \left(\frac{k}{n} \right)^{2/d} \frac{V_2}{V_0^{1+2/d}},
\]
to obtain 
\beq
\label{eq:mainknn}
W_{2,a} (\pi_{k,n}, \tilde{\mu})  \leq C (s + \sqrt{s} + I_1 + I_2 + s^{-1} \log(s) I_3) 
+ \sum_{m=4}^\infty \frac{C^m}{\sqrt{m! s^{m-1}}} I_m,
\eeq
where
\begin{itemize}
\item $I_1 = \sup_{i \in \{1,\dots,n\}} \left\|\frac{1}{k \step} \sum_{X_j \in \mathcal{B}(X_i,r_{\mathcal{X}_n}(X_i))}  (X_j - X_i) - f^{-2/d} \nabla \log f (X_i) \right\|$ ;
\item $I_2 = \sup_{i \in \{1,\dots,n\}} \left\|\frac{1}{2 k \step} \sum_{X_j \in \mathcal{B}(X_i,r_{\mathcal{X}_n}(X_i))}  (X_j - X_i)^{\otimes 2} -  \frac{f(X_i)^{-2/d} I_d}{2} \right\|$ ;
\item $\forall m > 2, I_m = \sup_{i \in \{1,\dots,n\}} \left\|\frac{1}{k} \sum_{X_j \in \mathcal{B}(X_i,r_{\mathcal{X}_n}(X_i))}  (X_j - X_i)^{\otimes m} \right\|$.
\end{itemize}

In the remainder of this Section, we show that, with probability greater than $1 - \frac{C}{n}$, 
\begin{enumerate}[label=(\roman*)]
\item \label{itm:ineq1} $I_1 \leq C \left( \sqrt{\frac{ \log n}{k}} \left(\frac{n}{k}\right)^{1/d} + \left(\frac{k}{n}\right)^{1/d}\right)$;
\item \label{itm:ineq2} $I_2 \leq C\left(\sqrt{\frac{\log n}{k}} + \left(\frac{k}{n}\right)^{2/d}\right)$;
\item \label{itm:ineq3} $s^{-1} I_3 \leq C \left(\sqrt{\frac{ \log n}{k}} \left(\frac{k}{n}\right)^{1/d} + \left(\frac{k}{n}\right)^{2/d}\right)$;
\item \label{itm:ineq4} $\forall m > 3,  I_m\leq C^m \left(\frac{k}{n}\right)^{m/d}$.
\end{enumerate}
Proposition~\ref{pro:randomgraph} is then obtained by injecting these bounds in (\ref{eq:mainknn}) and by remarking that, since $\mathcal{T}$ is compact and $a$ is smooth, then $W_{2,a} \leq C W_2$. 

Let $x \in \mathcal{T}, r > 0$, 
$N_r = \sum_{i=1}^n 1_{X_i \in \mathcal{B}(x,r)}$, $P_r = \mu(B(x,r))$ and 
$0 \leq \delta< 1$. By the multiplicative Chernoff bound,
\[
P\left(|N_r - n P_r| \geq \delta n P_r\right) \leq  2 e^{-\delta^2 n P_r / 3}.
\]
In particular, taking $\delta = \left(\frac{3\log(2n^2)}{n P_r}\right)^{1/2}$,
we obtain 
\beq
\label{eq:Chernoff}
P\left(|N_r - n P_r| \geq  (3 n P_r \log(2n^2))^{1/2}\right) \leq \frac{1}{n^2}.
\eeq
Then, taking $r_M = \left( \frac{2k}{n V_0 \min f} \right)^{1/d}$, we have $P_{r_M} \geq \frac{2k}{n}$ and 
\[
P\left(N_{r_M} \leq 2k  - C \sqrt{ k \log n}\right) \leq \frac{1}{n^2}.
\]
With $\log n < \frac{k}{C^2}$, we have that $\PP(N_{r_M} \geq k) \geq 1 - \frac{1}{n^2}$ and
\beq
\label{eq:radiusMax}
\PP(\radius_{\mathcal{X}_n}(x) \leq r_M) \geq 1 - \frac{1}{n^2}
\eeq
 and, by
a union bound, 
\[
\PP(\sup_{x \in \mathcal{X}_n} \radius_{\mathcal{X}_n}(x) \leq r_M) \geq 1 - \frac{1}{n}.
\]
Therefore, for all $m > 3$, 
\begin{align*}
 I_m & \leq  \frac{1}{k} \sum_{X_j \in \mathcal{B}(X_i,\tilde{r})}  \left\|X_j - X_i\right\|^m \\
& \leq r_M^m
\end{align*}
with probability $1 - \frac{1}{n}$ and inequality \ref{itm:ineq4} follows. 

Let us now prove inequality \ref{itm:ineq1}. Let $x \in \mathcal{T}$ and $r = \left( \frac{k}{n V_0 f(x)} \right)^{1/d}$. Using a Taylor expansion, we obtain
\begin{align*}
\EE[(X_i - x)& 1_{X_i \in \mathcal{B}(x,r)}] \\
& = \int_{\mathcal{B}(x,r)} (y-x) \mu(dy) \\
& =  \int_{\mathcal{B}(x,r)} (y-x)f(y) dy \\
& =  \int_{\mathcal{B}(x,r)} (y-x)f(x) + (y-x)^{\otimes 2} \nabla f(x) + \frac{(y-x)^{\otimes 3} \nabla^2 f (x)}{2} + O(r^4) \, dy.
\end{align*}
Hence, by symmetry of $\mathcal{B}(x,r)$, we have
\[
\EE[(X_i - x)1_{X_i \in \mathcal{B}(x,r)}] =  V_2 r^{d+2} \nabla f(x) + O(r^{d+4})
\]
and, by definition of $s$,
\beq
\label{eq:auxknn1}
\EE[(X_i - x)1_{X_i \in \mathcal{B}(x,r)}]= \frac{k s }{n f^{2/d}(x)} \nabla \log f (x) + O\left(s \left(\frac{k}{n} \right)^{1+2/d} \right).
\eeq
Let $b_1 = \frac{1}{k s} \sum_{i=1}^n (X_i - x )1_{X_i \in \mathcal{B}(x,r)}$. 
Since 
\[
\|(X_i - x )1_{X_i \in \mathcal{B}(x,r)} \| \leq r \leq C \left(\frac{k}{n}\right)^{1/d}
\]
and 
\[
\EE[\|X_i - x \|^21_{X_i \in \mathcal{B}(x,r)} ] \leq r^2 P_r \leq C \left( \frac{k}{n} \right)^{1+2/d}, 
\]
we can apply Bernstein's inequality to each coordinate of $b_1$, to obtain that 
\[
\forall t > 0, P(\left\|ks b_1 - ks \EE[b_1]\|_\infty \geq t \right) \leq 2 d e^{- \frac{t^2}{ C (k (k/n)^{2/d} + t(k/n)^{1/d})}}.
\]
Thus, as long as $k \geq 2C \log(2dn^2)$, taking $t = \sqrt{2 C \log(2 d n^2) k} \left(\frac{k}{n}\right)^{1/d}$ yields 
\[
P\left(\left\|ks b_1 - n \EE[(X_i - x)1_{X_i \in \mathcal{B}(x,r)}] \right\|_\infty \geq C \left( \frac{k}{n} \right)^{1/d} \sqrt{k \log n}\right) \leq \frac{1}{n^2}
\]
or 
\[
P\left(\left\|b_1 - \frac{n}{ks} \EE[(X_i - x)1_{X_i \in \mathcal{B}(x,r)}] \right\|_\infty \geq C \left(\sqrt{\frac{ \log n}{k}} \left(\frac{n}{k}\right)^{1/d}\right) \right) \leq \frac{1}{n^2}.
\]
Thus, by (\ref{eq:auxknn1}), 
\beq
\label{eq:auxknn2}
P\left(\left\|b_1 - f^{-2/d}(x) \nabla \log f (x) \right\|_\infty \geq C \left( \sqrt{\frac{ \log n}{k}} \left(\frac{n}{k}\right)^{1/d} + \left(\frac{k}{n} \right)^{2/d} \right) \right) \leq \frac{1}{n^2}. 
\eeq
Now, since $\|\nabla^2 f\|$ is bounded on $\mathcal{T}$, 
\begin{align*}
|P_r - V_0 r^{d}| = \left|\int_{B(x,r)} f(y) - f(x) dy \right | & \leq \int_{B(x,r)} r^2 \max_{y \in \mathcal{T}} \|\nabla^2 f(y)\| dy \\
&\leq V_2 r^{d+2} \max_{y \in \mathcal{T}} \|\nabla^2 f(y)\|,
\end{align*}
thus
\[
\left|P_r - \frac{k}{n}\right| \leq C\left(\frac{k}{n}\right)^{1+2/d}.
\]
Then, by (\ref{eq:Chernoff}), 
\beq
\label{eq:radiusconc}
\PP\left(|N_r - k| \leq C \left(\sqrt{k \log n}+\frac{k^{1+2/d}}{n^{2/d}}\right)\right) \geq 1 - \frac{1}{n^2}.
\eeq
By (\ref{eq:radiusMax}), taking $b_2 = \frac{1}{ks} \sum_{X_i \in \mathcal{B}(x,\radius_{\mathcal{X}_n}(x))} X_i - x$, we have with probability $1 - \frac{1}{n^2}$
\[
\|b_2 - b_1\| \leq \frac{r_M}{ks} |N_r - N_{\radius_{\mathcal{X}_n}(x)}| = \frac{r_M}{ks} |N_r - k| \leq C \frac{n^{1/d}}{k^{1+1/d}} |N_r - k|.
\]
Combining this bound with (\ref{eq:radiusconc}), we obtain  
\beq
\label{eq:auxknn3}
P\left(\|b_1 - b_2\| \leq C \left( \sqrt{\frac{ \log n}{k}} \left(\frac{n}{k}\right)^{1/d} + \left( \frac{k}{n} \right)^{1/d} \right)\right) \geq 1 - \frac{2}{n^2}.
\eeq
Combining (\ref{eq:auxknn2}) and (\ref{eq:auxknn3}), we have, with probability $1 - \frac{3}{n^2}$, 
\begin{align*}
\bigg\|\frac{1}{k \step} \sum_{X_i \in \mathcal{B}(x,r_{\mathcal{X}_n}(x))}  (X_i - x) & - f^{-2/d} \nabla \log f \bigg\| \\
& = \left\|b_2  - f^{-2/d}\nabla \log f \right\| \\
& \leq d\left\|b_1-f^{-2/d}\nabla \log f \right\|_\infty + C \left( \frac{n^{1/d} \sqrt{\log n} }{k^{1/2+1/d}} + \left( \frac{k}{n} \right)^{1/d} \right) \\
& \leq  C \left( \sqrt{\frac{ \log n}{k}} \left(\frac{n}{k}\right)^{1/d} + \left(\frac{k}{n}\right)^{1/d}\right).
\end{align*}
Inequality \ref{itm:ineq1} is finally obtained by using a union-bound. 

Let us derive inequalities \ref{itm:ineq2} and \ref{itm:ineq3} through similar computations. First, using a Taylor expansion, we obtain  
\[ 
\EE[(X_i - x)^{\otimes 2} 1_{X_i \in \mathcal{B}(x,r)}] = V_2 r^{d+2} f(x) I_d + O(r^{d+4}) 
\]
and 
\[
\EE[(X_i - x)^{\otimes 3} 1_{X_i \in \mathcal{B}(x,r)}] = O(r^{d+4}) 
\] 
Letting $a_1 = \frac{1}{2ks} \sum_{i=1}^n (X_i - x)^{\otimes 2} 1_{X_i \in \mathcal{B}(x,r)}$ and 
$c_1 = \frac{1}{ks} \sum_{i=1}^n (X_i - x)^{\otimes 3} 1_{X_i \in \mathcal{B}(x,r)}$, we have, by Bernstein's inequality, 
\[
P\left(\left\|a_1 - \frac{1}{2f^{2/d}(x)} I_d \right\| \geq C \left( \sqrt{\frac{ \log n}{k}} + \left(\frac{k}{n} \right)^{2/d} \right) \right) \leq \frac{1}{n^2} 
\]
and 
\[
P\left(\left\|c_1 \right\| \geq C \left(  \sqrt{\frac{ \log n}{k}} \left(\frac{k}{n}\right)^{1/d} + \left(\frac{k}{n} \right)^{2/d} \right) \right) \leq \frac{1}{n^2}.
\]
Then, letting $a_2 = \frac{1}{2ks} \sum_{i=1}^n (X_i - x)^{\otimes 2} 1_{X_i \in \mathcal{B}(x,r_{\mathcal{X}_n}(x))}$ and 
$c_2 = \frac{1}{ks} \sum_{i=1}^n (X_i - x)^{\otimes 3} 1_{X_i \in \mathcal{B}(x,r_{\mathcal{X}_n}(x))}$ and using (\ref{eq:radiusconc}) once more, 
\[
P\left(\|a_2 - a_1\| \leq  C\left( \sqrt{\frac{ \log n}{k}} + \left(\frac{k}{n} \right)^{2/d} \right)\right) \geq 1 - \frac{1}{n^2}
\]
and
\[
P\left(\|c_2 - c_1\| \leq C \left(\sqrt{\frac{ \log n}{k}} \left(\frac{k}{n}\right)^{1/d} + \left(\frac{k}{n}\right)^{3/d}\right)\right) \geq 1 - \frac{1}{n^2}
\]
From here, we obtain
\[
 \left\|\frac{1}{k \step} \sum_{X_i \in \mathcal{B}(x,r_{\mathcal{X}_n}(x))}  \frac{(X_i - x)^{\otimes 2}}{2} -  \frac{f(x)^{-2/d} I_d}{2} \right\|  \leq 
 C\left( \sqrt{\frac{ \log n}{k}} + \left(\frac{k}{n} \right)^{2/d} \right)
\]
and 
\[
 \left\|\frac{1}{k \step} \sum_{X_i \in \mathcal{B}(x,r_{\mathcal{X}_n}(x))}  (X_i - x)^{\otimes 3} \right\|  \leq 
 C \left(\sqrt{\frac{ \log n}{k}} \left(\frac{k}{n}\right)^{1/d} + \left(\frac{k}{n}\right)^{2/d}\right)
\]
with probability $1 - \frac{6}{n^2}$. 
Inequalities \ref{itm:ineq2} and \ref{itm:ineq3} are finally obtained thanks to a union bound inequality. 



\section{Approximation arguments}
\label{sec:approx}
Let us present the approximation arguments necessary to conclude the proofs of Theorem  \ref{thm:main}. 
Suppose Assumptions~\ref{ass:main} and \ref{ass:mainnu} are verified and let $s > 0$. Furthermore, we will assume, without any loss of generality, that $0 \in E$. We denote by $X$ a random variable drawn from the measure $\nu$ and by $Y$ the random variable corresponding to the state of the random walk with Markov kernel $K$ and initial state $X$ after a single jump. For any $t \in \RR$, let 
\[
\forall t \in \RR, Y(t) = \begin{cases} X \text{ if $t < \tau$}\\
Y \text{ otherwise}
\end{cases}
\]
Let $(K_n)_{n \in \NN}$ be a family of compact sets such that, for any $n \in \NN$, $K_n \subset K_{n+1} \subset E$ and $\cup_{n \in \NN} K_n = E$.  Let $0 < \epsilon_1, \epsilon_2 < 1$, $Z$ be a random variable drawn from $\mu$, $N$ be a random variable in the ball of radius $1$ with smooth density and let $I$ be a Bernoulli random variable with parameter $\epsilon_2$. These random variables are such that $N,I,Z, X$ and $Y$ are independent. Finally, let $U = \epsilon_1 N$. 
For $t \geq 0$, let
\[
\tilde{Y}(t) = IZ + (1-I) ( U + Y(t) 1_{X,Y \in K_n})
\]
and let $\tilde{\nu}_{n, \epsilon_1, \epsilon_2}$ be the measure of $\tilde{Y}(0)$. For any $x \in E, \test \in C^\infty_c(E, \RR)$ and $t> 0$, let 
\[
(\mathcal{L}_{\tilde{\nu}_{n, \epsilon_1, \epsilon_2}})_t \test (x)=  \EE\left[I \mathcal{L}_\mu \test (\tilde{Y}(t)) + \frac{(1-I)}{\step}[ \test(\tilde{Y}(t)) - \test(\tilde{Y}(0))]  \mid \tilde{Y}(0) = x\right].
\]
Let $t > 0$. Since $\tilde{Y}(t)$ and $\tilde{Y}(0)$ follow the same law and since $\EE[\mathcal{L}_\mu \test (Z)]= 0$, we have
\[
\EE[(\mathcal{L}_{\tilde{\nu}_{n, \epsilon_1, \epsilon_2}})_t \test (\tilde{Y}(0))] = 0. 
\]
Rewriting $(\mathcal{L}_{\tilde{\nu}_{n, \epsilon_1, \epsilon_2}})_t$, we obtain
\[
\EE[(\mathcal{L}_{\tilde{\nu}_{n, \epsilon_1, \epsilon_2}})_t \test (\tilde{Y}(0))]  =  \epsilon_2 \EE[\mathcal{L}_\mu \test(Z) ] + \frac{1- \epsilon_2}{\step} \EE\left[(\test(Y(t) + U) - \test(X + U))1_{ X,Y(t) \in K_n}\right].
\]
Then, for any bounded real analytic function $\test$ with bounded derivatives of all orders, 
\[
\EE[( \mathcal{L}_{\tilde{\nu}_{n, \epsilon_1, \epsilon_2}})_t \test (\tilde{Y}(0))] = \epsilon_2 \EE[\mathcal{L}_\mu \test(Z) ]  
+  (1- \epsilon_2) \EE\left[1_{X,Y \in K_n}   \sum_{k=1}^\infty \frac{1}{sk!} \left< (Y(t) - X)^{\otimes k}, \nabla^k \test(X + U) \right >\right]
\]
and
\begin{multline*}
\frac{\EE[((\mathcal{L}_{\tilde{\nu}_{n, \epsilon_1, \epsilon_2}})_t - \mathcal{L}_\mu) \test (\tilde{Y}(0))]}{1 - \epsilon_2} = \EE\left[1_{X,Y \in K_n}   \sum_{k=1}^\infty \frac{1}{\step k!} \left< (Y(t) - X)^{\otimes k}, \nabla^k \test(X + U) \right>\right] \\
- \EE\left[1_{X,Y \in K_n}\mathcal{L}_\mu \test(U + X) + 1_{(X,Y) \notin K_n^2} \mathcal{L}_\mu \test(U)\right] 
\end{multline*}
Since $K_n$ is a compact set, there exists a compact set $K_n' \subset E$ and $e(n) > 0$ such that, if $\epsilon_1 < e(n)$, then $(X + U)1_{X \in K_n} \in K_n'$. 
Let us now assume that $\epsilon_1 < e(n)$. 
Then, $\tilde{\nu}_{n, \epsilon_1, \epsilon_2}$ admits a measure $h$ with respect to $\mu$ such that $h = \epsilon + f$ with $\epsilon > 0$ and $f \in C^\infty_c(E,\RR)$. 
Thus, for $t > 0$, we can follow the computations of Section~\ref{sec:main} to obtain 
\[
I_\mu((\tilde{\nu}_{n, \epsilon_1, \epsilon_2})_t) \leq \sqrt{\EE[\tilde{S}(t)^2] I_\mu(\nu_t)}, 
\]
where 
\begin{align*}
\tilde{S}&(t)  = 1_{(X,Y) \notin K_n^2}(f_1(t) \|b(U)\|_{a^{-1}(U)}  + f_2(t) \sqrt{d}) \\
&  f_1(t) \left\|\EE\left[1_{X,Y \in K_n} \left(\frac{Y(t)-X}{s} - b(X + U)\right) \mid X + U\right]  \right\|_{a^{-1}(X + U)}  \\
& +  f_2(t) \left\|\EE\left[1_{X,Y \in K_n}\left(\frac{(Y(t)-X)^{\otimes 2}}{2s} - a(X + U) \right)\mid X + U\right] \right\|_{a^{-1}(X + U)} \\
& + \sum_{k=3}^{\infty} \frac{f_k(t)}{s k!} \left\|\EE[1_{ X,Y \in K_n} (Y(t)-X)^{\otimes k} \mid X + U]\right\|_{a^{-1}(X + U)},
\end{align*}
with the functions $(f_k)_{k \geq 1}$ defined as in Proposition~\ref{pro:curvdim}. We thus have
\[
I_\mu((\tilde{\nu}_{n, \epsilon_1, \epsilon_2})_t)^{1/2} \leq \EE[\tilde{S}(t)^2]^{1/2} 
\]
and, following the computations of Section~\ref{sec:main}, we obtain
\beq
\label{eq:approxImain}
(1 - c e^{- \kappa T}) W_{2,a}(\tilde{\nu}_{n, \epsilon_1, \epsilon_2}, \mu) \leq  \int_{0}^T \EE[\tilde{S}(t)^2]^{1/2}\, dt.
\eeq

In the following, we denote by $C$ a generic constant and by $C(n)$ a generic constant depending only on $n$. 
Let $t \in [0,T]$ and 
\begin{align*}
R_1(t) & =  f_1(t) \left\|\EE\left[1_{X,Y \in K_n} \left( \frac{Y(t)-X}{s} -  b(X) \right) \mid X + U\right]  \right\|_{a^{-1}(X + U)}  \\
& +  f_2(t) \left\|\EE\left[1_{X,Y \in K_n} \left(\frac{(Y-X)^{\otimes 2}}{2s} - a(X) \right) \mid X + U\right] \right\|_{a^{-1}(X + U)}  \\
& + \sum_{k=3}^{\infty} \frac{f_k(t)}{s k!} \left\|\EE[1_{X,Y \in K_n} (Y(t)-X)^{\otimes k} \mid X + U]\right\|_{a^{-1}(X + U)} .
\end{align*}
By the triangle inequality and by Jensen's inequality, we have 
\begin{align*}
\EE[\tilde{S}(t)^2]^{1/2} & - \EE[R_1(t)^2]^{1/2} \leq f_1(t) \EE[\|b(U)\|_{a^{-1}(U)}^2 1_{(X,Y) \notin K_n^2}]^{1/2}  +f_2(t) \sqrt{d} P((X,Y) \notin K_n^2)^{1/2} \\
& + f_1(t) \EE[1_{X,Y \in K_n} \| b(X + U) - b(X) \|^2_{a^{-1}(X + U)}]^{1/2} \\
& + f_2(t) \EE[1_{X,Y \in K_n}\|a(X + U) - a(X) \|^2_{a^{-1}(X + U)}]^{1/2}. 
\end{align*}
Since $\epsilon_1 < e(1)$, we have that $U, X$ and $X+U$ belongs to a compact subset of $E$ as long as $X \in K_n$. Since $b$ and $a$ are infinitely differentiable on $E$, they are in particular Lipschitz continuous on these compact sets. Thus, 
\[
\EE[\tilde{S}(t)^2]^{1/2} - \EE[R_1(t)^2]^{1/2} \leq (f_1(t) + f_2(t)) \left(\epsilon_1 C(n) + C P((X,Y) \notin K_n^2)^{1/2} \right)
\]
and
\beq
\label{eq:Sbound1}
\int_0^T \EE[\tilde{S}(t)^2]^{1/2} \, dt - \int_0^T \EE[R_1(t)^2]^{1/2} \, dt \leq  C(n) \epsilon_1  + C P((X,Y) \notin K_n^2)^{1/2}. 
\eeq
Now, let 
\begin{align*}
R_2(t) & =  f_1(t) \left\|\EE\left[1_{X,Y \in K_n}\left(\frac{Y(t)-X}{s}   - b(X)\right) \mid X\right]  \right\|_{a^{-1}(X + U)} \\
& +  f_2(t) \left\|\EE\left[1_{X,Y \in K_n}\left(\frac{(Y(t)-X)^{\otimes 2}}{2s}  - a(X)\right)\mid X\right] \right\|_{a^{-1}(X + U)} \\
& + \sum_{k=3}^{\infty} \frac{f_k(t)}{s k!} \left\|\EE[1_{X,Y \in K_n} (Y(t)-X)^{\otimes k} \mid X]\right\|_{a^{-1}(X + U)}.
\end{align*}
Since $U$ and the process $(Y(t))_{t \geq 0}$ are independent, we have, by Jensen's inequality,
\begin{align*}
\left\|1_{ X,Y \in K_n} \EE[(Y(t)-X)^{\otimes k} \mid X+U]\right\|_{a^{-1}(X + U)}  \\
&  \hspace{-2cm} = \left\| \EE[\EE[1_{X,Y \in K_n} (Y(t)-X)^{\otimes k}  \mid X]\mid X+U]\right\|_{a^{-1}(X + U)}  \\
&  \hspace{-2cm} \leq \EE\left[\left\|\EE[1_{X,Y \in K_n} (Y(t)-X)^{\otimes k}  \mid X]\right\|_{a^{-1}(X + U)} \mid X+U\right].
\end{align*}
As a similar inequality holds for other terms of $R_1$, we have 
\beq
\label{eq:Sbound3}
\EE[R_1(t)^2]^{1/2} \leq \EE[R_2(t)^2]^{1/2}. 
\eeq 
Now, let 
\begin{align*}
R_3(t) & =  f_1(t) \left\|\EE\left[1_{X,Y \in K_n}\left( \frac{Y(t)-X}{s}  - b(X)\right) \mid X\right]  \right\|_{a^{-1}(X)}\\
& +  f_2(t) \left\|\EE\left[1_{X,Y \in K_n}\left(\frac{(Y(t)-X)^{\otimes 2}}{2s}  - a(X)\right)\mid X\right] \right\|_{a^{-1}(X)} \\
& + \sum_{k=3}^{\infty} \frac{f_k(t)}{s k!} \left\|\EE[1_{ X,Y \in K_n} (Y(t)-X)^{\otimes k} \mid X ]\right\|_{a^{-1}(X)}.
\end{align*}
For $x,y \in K_n'$,  $k \in \NN^\star$ and $u \in (\RR^d)^{\otimes k}$, we have  
\begin{align*}
\|u\|^2_{a^{-1}(y)} - \|u\|^2_{a^{-1}(x)} & = \sum_{i,j \in \{1,\dots,d\}^k} u_i u_j (\Pi_{l=1}^k a^{-1}_{i_l,j_l}(y) - \Pi_{l=1}^k a^{-1}_{i_l,j_l}(x)) \\
& \leq \|u\|^2 \left(\sum_{i,j \in \{1,\dots,d\}^k} \left(\Pi_{l=1}^k a^{-1}_{i_l,j_l}(y) - \Pi_{l=1}^k a^{-1}_{i_l,j_l}(x)\right)^2 \right)^{1/2} \\
& \leq d^k \|u\|^2 \sup_{i,j \in \{1,\dots,d\}^k} \left|\Pi_{l=1}^k a^{-1}_{i_l,j_l}(y) - \Pi_{l=1}^k a^{-1}_{i_l,j_l}(x) \right|.
\end{align*}
Since $a^{-1}$ is infinitely differentiable on $E$, its coefficients are bounded and have bounded derivatives on the compact set $K'_n$. Hence, there exists $C(n)$ such that, for any $i,j \in \{1,\dots,d\}^{k}$,
\begin{align*}
\Pi_{l=1}^k a^{-1}_{i_l,j_l}(y) - \Pi_{l=1}^k a^{-1}_{i_l,j_l}(x) & \leq  \sup_{\xi \in K_n} \| \nabla \Pi_{l=1}^k a^{-1}_{i_l,j_l}(\xi)\| \|y-x\| \\
& \leq k \left\|\sup_{\xi \in K_n, i,j \in \{1,\dots,d\}} \nabla a^{-1}_{i,j} \right\| \left|\sup_{\xi \in K_n, i,j \in \{1,\dots,d\}} a^{-1}_{i,j}\right|^{k-1} \|y-x\| \\
& \leq C(n)^k \|y-x\|
\end{align*}
leading to 
\[
\|u\|^2_{a^{-1}(y)} - \|u\|^2_{a^{-1}(x)} \leq \|u\|^2 C(n)^k \|y-x\|.
\]
Therefore, by the triangle inequality and since $\|U\| \leq \epsilon_1$,
\begin{multline*}
\EE[R_2(t)^2]^{1/2} - \EE[R_3(t)^2]^{1/2} \leq \epsilon_1^{1/2} \bigg(f_1(t) \EE[C(n)  \|b(X)\|^2 1_{X,Y \in K_n}]^{1/2} \\ 
+ f_2(t)  \EE[ C(n)^2 \|a(X)\|^2 1_{X,Y \in K_n}]^{1/2}  
+ \frac{1}{\step} \EE\left[\left(\sum_{k=1}^\infty \frac{f_k(t) C(n)^k}{k!} \|Y(t) - X\|^{2k} 1_{X, Y \in K_n} \right)^2 \right]^{1/2}\bigg).
\end{multline*}
Now, if $t < \tau$, we have 
\[
\|Y(t) - X\| = \| X - X\| = 0.
\]
On the other hand, if $t \geq \tau$, 
\[
\|Y(t) - X\| 1_{X, Y \in K_n} = \|Y-X\| 1_{X, Y \in K_n} \leq 2D(n),
\]
where $D(n)$ is the diameter of the compact set $K_n$. Therefore, 
\beq
\label{eq:Sbound2}
\int_0^T \EE[R_2(t)^2]^{1/2} \, dt - \int_0^T \EE[R_3(t)^2]^{1/2} \, dt \leq C(n) \epsilon_1^{1/2}.
\eeq
Finally, 
\begin{multline*}
\EE[R_3(t)^2]^{1/2} - \EE[S(t)^2]^{1/2} \leq \EE\left[ \left(\sum_{k=1}^\infty \frac{f_k(t)}{k!} \|Y(t) - X\|_{a^{-1}(X)}^k  \right)^2 1_{(X,Y) \notin K_n^2}\right]^{1/2} \\
+ f_1(t) \EE[\|b(X)\|^2_{a^{-1}(X)} 1_{(X,Y) \notin K_n^2}]^{1/2} + f_2(t) \sqrt{d} \mathbb{P}((X,Y) \notin K_n^2) ,
\end{multline*}
where $S(t)$ is defined in (\ref{eq:Sdef}).

\begin{multline}
\int_0^T \EE[R_3(t)^2]^{1/2} \, dt - \int_0^T \EE[S(t)^2]^{1/2} \, dt \leq T \EE\left[  \sup_{t \in [\tau, T]} \left(\sum_{k=1}^\infty \frac{f_k(t)}{k!} \|Y(t) - X\|_{a^{-1}(X)}^k  \right)^2 1_{(X,Y) \notin K_n^2}\right]^{1/2} \\
+ C\left(\EE[\|b(X)\|^2_{a^{-1}(X)} 1_{(X,Y) \notin K_n^2}]^{1/2} + \mathbb{P}((X,Y) \notin K_n^2)\right).
\end{multline}
We can assume that $\|b\|_{a{-1}} \in L_2(\nu)$, otherwise the bound of Theorem \ref{thm:main} is just a trivial one. Thus, by \ref{ass:mainnu3}, we have that there exists $L(n)$ such that $\lim \limits_{n \rightarrow \infty} L(n) = 0$ and 
\beq
\label{eq:Sbound4}
\int_0^T \EE[R_3^2]^{1/2} \, dt - \int_0^T \EE[S(t)^2]^{1/2} \, dt \leq L(n).
\eeq
By combining (\ref{eq:Sbound1}), (\ref{eq:Sbound3}), (\ref{eq:Sbound2}) and (\ref{eq:Sbound4}), we finally obtain that there exists  $C, C(n) > 0$ such that 
\beq
\label{eq:Sboundfinal}
\int_0^T \EE[\tilde{S}(t)^2]^{1/2} \, dt \leq \int_0^T \EE[S(t)^2]^{1/2} \, dt + C(n) \epsilon_1^{1/2} + C P((X,Y) \notin K_n^2)^{1/2} + L(n).  
\eeq
Letting $\epsilon_1$ go to zero, $n$ go to infinity and $\epsilon_2$ go to zero thus yields 
\[
\lim\limits_{ \epsilon_2 \rightarrow 0} \lim\limits_{ n \rightarrow \infty} \lim\limits_{ \epsilon_1 \rightarrow 0} W_{2,a}(\tilde{\nu}_{n, \epsilon_1, \epsilon_2}, \mu) \leq \int_0^T \EE[S(t)^2]^{1/2} dt
\]
and, since $d_a(.,0) \in L_2(\nu) \cap L_2(\mu)$, then, by Theorem 6.9 \citep{Villani}, 
\[
\lim\limits_{ \epsilon_2\rightarrow 0} \lim\limits_{ n \rightarrow \infty} \lim\limits_{ \epsilon_1 \rightarrow 0} W_{2,a}(\tilde{\nu}_{n, \epsilon_1, \epsilon_2}, \nu) = 0
\]
which concludes the proof since 
\[
W_{2,a}(\nu, \mu) \leq \lim\limits_{ \epsilon_2 \rightarrow 0} \lim\limits_{ n \rightarrow \infty} \lim\limits_{ \epsilon_1 \rightarrow 0} W_{2,a}(\nu, \tilde{\nu}_{n, \epsilon_1, \epsilon_2}) + W_{2,a}(\tilde{\nu}_{n, \epsilon_1, \epsilon_2}, \mu) \leq \int_0^T \EE[S(t)^2]^{1/2} dt.
\]

\bibliographystyle{plain}
\bibliography{Bibliography}

\begin{thebibliography}{10}

\bibitem{Markov}
Dominique Bakry, Ivan Gentil, and Michel Ledoux.
\newblock {\em {Analysis and Geometry of Markov Diffusion operators}}.
\newblock Grundlehren der mathematischen Wissenschaften, Vol. 348. {Springer},
  Jan 2014.

\bibitem{eigenmap}
Mikhail Belkin and Partha Niyogi.
\newblock Laplacian eigenmaps for dimensionality reduction and data
  representation.
\newblock {\em Neural Comput.}, 15(6):1373--1396, June 2003.

\bibitem{labelprop}
Mikhail Belkin and Partha Niyogi.
\newblock Semi-supervised learning on riemannian manifolds.
\newblock {\em Machine Learning}, 56(1):209--239, 2004.

\bibitem{Bonis}
Thomas Bonis.
\newblock Stein’s method for normal approximation in wasserstein distances
  with application to the multivariate central limit theorem.
\newblock {\em Probability Theory and Related Fields}, 178, 12 2020.

\bibitem{Braverman1}
A.~{Braverman} and J.~G. {Dai}.
\newblock {Stein's method for steady-state diffusion approximations of
  $M/Ph/n+M$ systems}.
\newblock {\em ArXiv e-prints}, March 2015.

\bibitem{Braverman3}
A.~{Braverman} and J.~G. {Dai}.
\newblock {High order steady-state diffusion approximation of the Erlang-C
  system}.
\newblock {\em ArXiv e-prints}, February 2016.

\bibitem{Braverman2}
A.~{Braverman}, J.~G. {Dai}, and J.~{Feng}.
\newblock {Stein's method for steady-state diffusion approximations: an
  introduction through the Erlang-A and Erlang-C models}.
\newblock {\em ArXiv e-prints}, December 2015.

\bibitem{ethierkurtz}
Stewart~N. Ethier and Thomas~G. Kurtz.
\newblock {\em Markov processes : characterization and convergence}.
\newblock Wiley series in probability and mathematical statistics. J. Wiley \&
  Sons, New York, Chichester, 1986.

\bibitem{GGB}
Arnaud Guillin, Ivan Gentil, and François Bolley.
\newblock Convergence to equilibrium in wasserstein distance for fokker-planck
  equations.
\newblock {\em Journal of Functional Analysis}, 263(8):2430--2457, 2012.

\bibitem{Gurvich}
Itai Gurvich.
\newblock {Diffusion models and steady-state approximations for exponentially
  ergodic Markovian queues}.
\newblock {\em The Annals of Applied Probability}, 24(6):2527 -- 2559, 2014.

\bibitem{Hashimoto}
Tatsunori~B. Hashimoto, Yi~Sun, and Tommi~S. Jaakkola.
\newblock Metric recovery from directed unweighted graphs.
\newblock In {\em Proceedings of the Eighteenth International Conference on
  Artificial Intelligence and Statistics, {AISTATS} 2015, San Diego,
  California, USA, May 9-12, 2015}, 2015.

\bibitem{krantz2002primer}
S.G. Krantz and H.R. Parks.
\newblock {\em A Primer of Real Analytic Functions}.
\newblock Advanced Texts Series. Birkh{\"a}user Boston, 2002.

\bibitem{duality}
Kazumasa Kuwada.
\newblock Duality on gradient estimates and wasserstein controls.
\newblock {\em Journal of Functional Analysis}, 258(11):3758 -- 3774, 2010.

\bibitem{Stein}
Michel Ledoux, Ivan Nourdin, and Giovanni Peccati.
\newblock Stein's method, logarithmic sobolev and transport inequalities.
\newblock {\em Geometric and Functional Analysis}, 25(1):256--306, 2015.

\bibitem{Villani}
C\'edric Villani.
\newblock {\em Optimal transport : old and new}.
\newblock Grundlehren der mathematischen Wissenschaften. Springer, Berlin,
  2009.

\bibitem{spectral}
Ulrike von Luxburg.
\newblock A tutorial on spectral clustering.
\newblock {\em Statistics and Computing}, 17(4):395--416, 2007.

\bibitem{roadmap}
Ulrike von Luxburg and Morteza Alamgir.
\newblock Density estimation from unweighted k-nearest neighbor graphs: a
  roadmap.
\newblock In {\em Advances in Neural Information Processing Systems 26: 27th
  Annual Conference on Neural Information Processing Systems 2013. Proceedings
  of a meeting held December 5-8, 2013, Lake Tahoe, Nevada, United States.},
  pages 225--233, 2013.

\bibitem{Fwang}
F.-Y. {Wang}.
\newblock {Exponential Contraction in Wasserstein Distances for Diffusion
  Semigroups with Negative Curvature}.
\newblock {\em ArXiv e-prints}, March 2016.

\bibitem{WangOV}
Feng-Yu Wang.
\newblock Probability distance inequalities on riemannian manifolds and path
  spaces.
\newblock {\em Journal of Functional Analysis}, 206(1):167 -- 190, 2004.

\end{thebibliography}

\end{document}